\documentclass[12pt]{amsart}

\usepackage{amsmath}
\usepackage{latexsym,amsfonts,amssymb,mathrsfs}
\usepackage[all,cmtip]{xy}

\usepackage[
textwidth=3cm,
textsize=small,
colorinlistoftodos]
{todonotes}
\usepackage{tikz-cd}
\tikzset{shorten <>/.style={shorten >=#1,shorten <=#1}}
\tikzset{%
    symbol/.style={%
        ,draw=none
        ,every to/.append style={%
            ,edge node={%
                node [%
                    ,sloped
                    ,allow upside down
                    ,auto=false
                    ]{$#1$}
                }
            }
        }
    }

\newcommand{\cC}{\mathcal{C}}
\newcommand{\cD}{\mathcal{D}}

\newcommand{\cJ}{\mathcal{J}}

\newcommand{\cS}{\mathcal{S}}

\newcommand{\cX}{\mathcal{X}}

\newcommand{\cat}{\cC\!\mathit{at}}

\DeclareMathOperator*{\colim}{colim}

\newcommand{\C}{\mathcal{C}}

\numberwithin{equation}{section}

\theoremstyle{plain}
\newtheorem{thm}[equation]{Theorem}
\newtheorem*{thm*}{Theorem}
\newtheorem{cor}[equation]{Corollary}
\newtheorem{lem}[equation]{Lemma}
\newtheorem{prop}[equation]{Proposition}
\newtheorem*{prop*}{Proposition}

\theoremstyle{definition}
\newtheorem{defn}[equation]{Definition}

\theoremstyle{remark}
\newtheorem{rem}[equation]{Remark}
\newtheorem{ex}[equation]{Example}
\newtheorem{notation}[equation]{Notation}

\begin{document}
\title{Goodwillie Towers of $\infty$-categories and Desuspension}
\author[D.~Fuentes-Keuthan]{Daniel Fuentes-Keuthan}
\address{Department of Mathematics,
Johns Hopkins University, Baltimore
}
\email{danielfk@jhu.edu}
\address{}
\date{}
\begin{abstract}
   We reconceptualize the process of forming $n$-excisive approximations to $\infty$-categories, in the sense of Heuts, as inverting the suspension functor lifted to $A_n$-cogroup objects.  We characterize $n$-excisive $\infty$-categories as those $\infty$-categories in which $A_n$-cogroup objects admit desuspensions. Applying this result to pointed spaces we reprove a theorem of Klein-Schwänzl-Vogt: every 2-connected cogroup-like $A_\infty$-space admits a desuspension. 
\end{abstract}
\maketitle
\tableofcontents


\section*{Introduction}
Given a suitably nice $\infty$-category $\C$ one may form the stablization $Sp(\C)$ by formally inverting the suspension functor $\Sigma:\C\to \C$, as given by the filtered colimit \[\C\xrightarrow{\Sigma}\C\xrightarrow{\Sigma}\dots\C\xrightarrow{\Sigma}\dots\to Sp(\C) \]
This new category $Sp(\C)$ is easier to work with: it is the homotopic analogue of the category of abelian group objects in $\C$ and satisfies formal properties which make it into a sort of linearization of $\C$.  This stabilization procedure however loses a lot of information about $\C$, since, for instance, in $Sp(\C)$ we can recover an object from only its suspension. One might remedy this harsh truncation by requiring additional structure on the suspension in order to recover the object.  In particular the suspension functor lifts to a functor landing in $A_\infty$-cogroup objects in $\C$. \[\Sigma: \C\to coGrp_\infty(\C) \]
By forgetting down to the $A_n$-cogroup structures for each $n$ one obtains a collection of functors \[\Sigma: \C\to coGrp_n(\C) \] which can be formally inverted by taking the filtered colimit of \[ \C\xrightarrow{\Sigma} coGrp_n(\C)\xrightarrow{\Sigma}coGrp_n(coGrp_n(\C))\xrightarrow{\Sigma}\cdots \] to give a tower of approximations to $\C$ which capture more information as $n$ grows larger.  In particular $coGrp_1(\C)\simeq \C$ so the first stage of this process is just the stabilization $Sp(\C)$.

In \cite{H} Heuts introduces another tower of $\infty$-categories which interpolates between $\C$ and $Sp(\C)$ \[\dots\xrightarrow{\Sigma_{n+2, n+1}}P_{n+1}\C\xrightarrow{\Sigma_{n+1, n}}P_n\C\xrightarrow{\Sigma_{n, n-1}}\dots P_1\C \simeq Sp(\C)\]
via the $n$-excisive approximations $P_n(\C)$ to $\C$. We refer to this tower as the Goodwillie tower of $\C$. The $\infty$-category $P_n(\C)$ is built by formally inverting the suspension functor viewed as landing in an $\infty$-category $T_n(\C)$ of certain ``special" punctured $(n+1)$-cubes.

The main goal of this paper is to prove that the two constructions above are the same by showing that the $\infty$-category $T_n(\C)$ gives a new model for $A_n$-cogroup objects in $\C$.  Namely we have Theorem \ref{cubeiscogroup}:

\begin{thm*} Let $\C$ be a pointed $\infty$-category with finite colimits, then there is an equivalence \[coGrp_n(\C)\overset{\sim}{\longrightarrow} T_n(\C)\]
\end{thm*}

As a corollary of this result we obtain a characterization of $n$-excsisive $\infty$-categories as those $\infty$-categories in which one can desuspend $A_n$-cogroup objects.  This is Theorem \ref{layer}:

\begin{thm*} Let $\C$ be a pointed, compactly generated $\infty$-category. Then $\C$ is $n$-excisive if and only if the  suspension functor induces an equivalence
\[\Sigma:\C\overset{\sim}{\longrightarrow} coGrp_n(\C) \]
\end{thm*}

Passing to the limit as $n$ goes to infinity of the Goodwillie tower, one obtains an $\infty$-category $P_\infty\C$ which best approximates $\C$ from the point of view of Goodwillie calculus.  Theorem \ref{layer} extends to this $\infty$-category to obtain an $A_\infty$ desuspension theorem, stated as Theorem \ref{limit}:

\begin{thm*}
Let $\C$ be a pointed, compactly generated $\infty$-category. The suspension functor induces an equivalence \[\Sigma: P_\infty\C\overset{\sim}{\longrightarrow} coGrp_\infty(P_\infty\C)\]
\end{thm*}

These results together motivate a philosophical claim that Goodwillie calculus attempts to reconstruct $\infty$-categories and functors via the $A_\infty$ cogroup structure on the suspension functor.

As an application, in this new light we may translate certain desuspension problems into problems of the convergence of the Goodwillie tower. For example, from knowing that the Goodwillie tower convergences on simply connected spaces, we are able to recover the following theorem of Klein-Schwänzl-Vogt, which we state as Theorem \ref{ksv}

\begin{thm*}
The suspension functor induces an equivalence of $\infty$-cat\-e\-gories
\[ \Sigma: \cS_*^{\ge 2} \overset{\sim}{\longrightarrow} coGrp_{\infty}(\cS_*^{\ge 3}) \]
In particular every 2-connected $A_{\infty}$-comonoid admits a desuspension to a 1-connected space.
\end{thm*}

A brief outline of this article is as follows:
\begin{enumerate}
    \item In Section \ref{equivofmodels} we develop the punctured cube model of $A_n$-cogroup objects and prove Theorem \ref{cubeiscogroup} which states that our model is equivalent to the existing model developed by Lurie in \cite[Section 4.1]{HA}.
    \item In Section \ref{goodwillie} we prove Theorem \ref{layer} and Theorem \ref{limit} which establishes the connection between Goodwillie calculus and co\-group objects.
    \item In Section \ref{spaces} we apply Theorem \ref{limit} to the special case of pointed spaces in order to recover Theorem \ref{ksv}: that all 2-connected $A_\infty$-cogroup objects admit a desuspension.
\end{enumerate}

\textbf{Acknowledgements.}
The author is grateful to his advisor Emily Riehl for all of her support during the completion of this project. This work was supported by the National Science Foundation grant DMS-1652600 and was completed during a visit to Utrecht University whom the author thanks for their generous hospitality. The idea that special punctured cubical diagrams could give a model for cogroup objects originated from Jacob Lurie. Finally, the author is immensely grateful to Gijs Heuts for suggesting this project as well as for providing considerable feedback, mentorship, and guidance towards it completion.

\section{Punctured Cubes as coGroups}\label{equivofmodels}
In this section we develop a new model for $A_n$-cogroup objects in an $\infty$-category in terms of certain punctured cubical diagrams, and prove that this model is equivalent to the model dual to the $A_n$-group model developed in \cite[Section 4.1]{HA}. The dual story holds for group objects as well, but we will not need this. We begin by recalling the established formalism of $A_n$-comonoid objects.

\begin{defn}
Let $\Delta^{inj}$ be the wide subcategory of the category $\Delta$ containing only the injections.  Let $\Delta_n \hookrightarrow \Delta^{inj}$ be the full subcategory on the objects $[0], [1],..., [n]$.
\end{defn}
\begin{rem}
Throughout this paper when we discuss maps in $\Delta$ we only ever consider the injections.
\end{rem}

\begin{defn}\label{comonoiddef}
Let $\C$ be a pointed $\infty$-category with finite colimits. A (unital) \textit{$A_n$-comonoid object} in $C$ is a functor $A: \Delta_n\rightarrow \C$ satisfying the conditions
\begin{enumerate}
    \item $A[0]$ is a zero object. 
    \item For each $t\le n$ the maps $p_k:[1]\rightarrow [t]$ sending 0 to 0 and 1 to $k$ induce an equivalence \[ \coprod_{i=1}^n A[1]\xrightarrow{\sim} A[t] \]
    \item The following composites are the identity on $A[1]$  \[ A[1]\xrightarrow{\partial_{1}} A[2]\simeq A[1]\amalg A[1] \rightarrow A[1]\amalg 0 \simeq A[1]  \] \[ A[1] \xrightarrow{\partial_{1}} A[2]\simeq A[1]\amalg A[1] \rightarrow  0 \amalg A[1] \simeq A[1]  \]
    where the right hand arrow is the coproduct of the identity of $A[1]$ and the zero map.
\end{enumerate}
Let $coMon_n(\C)\hookrightarrow Fun(\Delta_n, \C)$ denote the full subcategory spanned by the $A_n$-comonoid objects in $\C$.
\end{defn}

\begin{rem}
For $n=0, 1$ conditions $(2)$ and $(3)$ do not make sense.  By convention an $A_0$-comonoid object is a zero object, and an $A_1$-comonoid object is simply an object of $\C$. The second condition is dual to the usual Segal condition for a monoid, the third condition says that the comultiplication induced by the Segal conditions is counital. Because we are in a pointed setting we have only one choice of counit, greatly simplifying the data needed to specify counitality.
\end{rem}

\begin{defn}\label{cogroup}
An $A_n$-comonoid object $A$ in $\C$ is cogroup-like if it satisfies conditions 1 and 3 above, and the stronger condition $(2')$.
\begin{itemize}
    \item[(2')] For each $t\le n$ and partition of $[n]$ into subsets $S, S'$ agreeing on at most one element the diagram
    \begin{center}
        \begin{tikzcd}
        A[S\cap S']\ar[r]\ar[d]    &   A[S]\ar[d]    \\
        A[S']\ar[r]       & A[n]
        \end{tikzcd}
    \end{center}
    is a pushout square in $\C$.
\end{itemize}
In particular condition $(2')$ implies the usual Segal condition $(2)$.  We will denote by $coGrp_n(\C)$ the $\infty$-category of cogroup-like $A_n$-comonoids in $\C$.  We will refer to these objects as \textit{$A_n$-cogroup objects.}
\end{defn}

It is worth mentioning that the cogroup-like condition (2') can be replaced with a more natural condition involving all pushout squares in $\Delta_n$.
\begin{lem}\label{colimsegal}
A functor  $A: \Delta_n\rightarrow \C$ satisfies condition $(2')$ of Definition \ref{cogroup} if and only if it satisfies the following condition
\begin{itemize}
    \item[(2'')] For each $t\le n$ and any two subsets $S, S'$ of $[t]$ the diagram
        \begin{center}
            \begin{tikzcd}
            A[S]\cap A[S']\ar[d]\ar[r] & A[S]\ar[d]\\
            A[S']\ar[r] & A[S\cup S']
            \end{tikzcd}
        \end{center}
        is a pushout square in $\C$.
\end{itemize}
\end{lem}
\begin{proof}
This new condition is clearly stronger than $(2')$, so we will show that condition $(2')$ suffices. Consider a pushout square in $\Delta_n$
\begin{center}
    \begin{tikzcd}
    S\cap S'\ar[d]\ar[r] & S\ar[d]\\
    S'\ar[r] & S\cup S'
    \end{tikzcd}
\end{center}
If $S$ and $S'$ intersect in one element then we are done by condition $(2')$.  Otherwise they intersect in at least two objects, call one of them $i$ so that $S\cap S' - \{i\}$ is nonempty.  Consider the diagram

\begin{center}
    \begin{tikzcd}
    A[\{i\}]\ar[r]\ar[d]    &  A[S\cap S']\ar[r]\ar[d]  &  A[S]\ar[d] \\
    A[(S' - \{S\cap S' -\{i\}\})]\ar[r] & A[S']\ar[r] & A[S\cup S']
    \end{tikzcd}
\end{center}
By condition $(2')$ the outside and left hand squares are pushouts, so that the right hand square is also a pushout.
\end{proof}

\begin{rem}
The property of being cogroup like is detected already at the level of the homotopy category, see for example \cite[Proposition 2.3]{GGN} for a similar discussion, and does not require us to specify any additional structure, only properties.
\end{rem}

Now we will give our alternative model of $A_n$-cogroup objects in terms of punctured cubes.

\begin{defn}
Let $\square[n+1] = P(\{0,\dots,n\})$ denote the powerset on $n+1$ elements. An \textit{$(n+1)$-cube} in an $\infty$-category $\C$ is a functor \[\cX:\square[n+1]\rightarrow \C\]
\end{defn}

\begin{defn}
Let $\square[n+1]_0 = P_0(\{0,\dots, n\})$ denote the nerve of the powerset on $n+1$ elements except that we remove the empty set. A \textit{punctured $(n+1)$-cube} in $\C$ is a functor \[\cX:\square[n+1]_0\rightarrow \C\] In words it is an $(n+1)$-cube missing the initial vertex.
\end{defn}

\begin{defn}\label{puncturedcube}
A punctured $(n+1)$-cube $\cX:P_0(\{0,\dots,n\})\rightarrow \C$ is called \textit{special} if 
\begin{enumerate}
    \item For every one element subset $s\in P(\{0,\dots,n\})$ there is an equivalence $\cX(s)\simeq 0$.
    \item The cube is strongly cocartesian: every face which exists in the cube is a pushout.
\end{enumerate}
Following \cite{H} we let $T_n\C$ denote the $\infty$-category of special punctured $(n+1)$-cubes (note the change in index).
\end{defn}

We will find it convenient to occasionally work with a different description of $P_0(\{1,\dots,n\})$

\begin{lem}\label{sliceiso}
There is an isomorphism of categories \[ \square[n+1]_0 \simeq \Delta^{inj}/[n]\] between the punctured $(n+1)$-cube and the slice category of $\Delta^{inj}$ over the object $[n]$.
\end{lem}
\begin{proof}
The idea is that an injection $[k]\to [n]$ designates a unique subset of $[n]$, and this allows us to pass between the two categories.
\end{proof}

Under the above isomorphism there is a natural functor \[\pi_n: Fun(\Delta_n, \C)\to Fun(\square[n+1]_0, \C) \] given by restricting along the projection $\pi_n:\Delta^{inj}/[n]\rightarrow \Delta_n$, which by Lemma \ref{colimsegal} restricts to a functor \[\pi_n: coGrp_n(\C)\rightarrow T_n(\C)\] 

\begin{ex}
The functor $Fun(\Delta_2, \C)\to Fun(\square[3]_0, \C)$ blows up a  truncated cosimplicial object \[\begin{tikzcd}
A[0] \arrow[r, shift left, "\partial_0"]\ar[r, "\partial_1"'] & {A[1]} \arrow[r, "\partial_1" description] \arrow[r, "\partial_2"', shift right] \arrow[r, "\partial_0", shift left] & {A[2]}
\end{tikzcd}\] into the punctured 3-cube

\begin{center}
\begin{tikzcd}[row sep = 5pt, column sep = 5pt]
                        
                                                          &                                                          & {A[0]} \arrow[dd, near start, "\partial_0"] \arrow[rd,near start, "\partial_1"] &                                 \\
                                                          & {A[0]} \arrow[rr, near start,"\partial_1"] \arrow[dd, near start,"\partial_0"] &                                                          & {A[1]} \arrow[dd, "\partial_1"] \\
{A[0]} \arrow[rr, near start,"\partial_1"] \arrow[rd, near start,"\partial_0"'] &                                                          & {A[1]} \arrow[rd, "\partial_2"]                          &                                 \\
                                                          & {A[1]} \arrow[rr, "\partial_0"]                          &                                                          & {A[2]}                         
\end{tikzcd}
\end{center}


\end{ex}

For the remainder of this section we will establish our first main result, Theorem \ref{layer}, which states that the functor $\pi_n$ provides an equivalence of models of $A_n$-cogroups. Because the proof is not exactly enlightening, we give an explicit description of the $n=2$ case which contains our key intuition.

\begin{ex}[Theorem \ref{cubeiscogroup} for $n=2$]\label{case3} The data of a special punctured 3-cube is a diagram
\begin{center}
\begin{tikzcd}[column sep = 5pt, row sep=5pt]
                        &                         & 0 \arrow[rd] \arrow[dd] &              \\
                        & 0 \arrow[rr, crossing over] \arrow[dd, crossing over] &                         & A \arrow[dd] \\
0 \arrow[rd] \arrow[rr] &                         & A'' \arrow[rd]          &              \\
                        & A' \arrow[rr]           &                         & A'''        
\end{tikzcd}
\end{center}
with every face a pushout.  Consider the diagram formed by the bottom and right square faces.
\begin{center}
\begin{tikzcd}
0 \arrow[d] \arrow[r]                       & A'' \arrow[d, "a"] & 0 \arrow[d] \arrow[l]                      \\
A' \arrow[d, no head, equals] \arrow[r] & A''' \arrow[d]     & A \arrow[d, no head, equals] \arrow[l] \\
A' \arrow[r, dashed, "\sim"]                        & Cof(a)             & A \arrow[l, dashed, "\sim"']                      
\end{tikzcd}
\end{center}

Because the squares are pushout, taking cofibers yields the equivalences as displayed, giving an equivalence $A\simeq A'$.  By symmetry we have also $A\simeq A''$.  We can identify each non zero vertex with $A$ and $A\vee A$, as well as the bottom face with the coproduct inclusions.
\begin{center}
    \begin{tikzcd}[column sep = 5pt, row sep=5pt]
                        &                         & 0 \arrow[rd] \arrow[dd] &                   \\
                        & 0 \arrow[rr, crossing over] \arrow[dd, crossing over] &                         & A \arrow[dd, "m"] \\
0 \arrow[rd] \arrow[rr] &                         & A \arrow[rd, "i_1"]     &                   \\
                        & A \arrow[rr, "i_2"']    &                         & A\vee A          
\end{tikzcd}
\end{center}
Note that once we substitute these equivalences, we are left with the interesting right hand vertical map $m:A\to A\vee A$. Being strong cocartesian tells us that $m$ satisfies counital properties as demonstrated below
\begin{center}
\begin{tikzcd}
0 \arrow[d] \arrow[r]                           & A \arrow[d, "m"] \arrow[rdd, no head, Rightarrow, bend left] &   & 0 \arrow[d] \arrow[r]                           & A \arrow[d, "m"] \arrow[rdd, no head, Rightarrow, bend left] &   \\
A \arrow[r, "i_2"] \arrow[rrd, "0", bend right] & A\vee A \arrow[rd]                                           &   & A \arrow[r, "i_1"] \arrow[rrd, "0", bend right] & A\vee A \arrow[rd]                                           &   \\
                                                &                                                              & A &                                                 &                                                              & A
\end{tikzcd}
\end{center}
Where the unique induced maps on the above left and right are respectively $A\vee A\xrightarrow{A \vee \eta} A$ and $A\vee A\xrightarrow{\eta \vee A} A$.

By Definition \ref{comonoiddef} this gives the object \begin{tikzcd}
0 \arrow[r] & A \arrow[r, shift right] \arrow[r] \arrow[r, shift left] & A\vee A
\end{tikzcd}, obtained by collapsing the cube, the structure of an $A_2$-comonoid. Furthermore, because every face is a pushout, this comonoid object is cogroup like.

Now suppose we have a morphism in $T_2\C$ between the two punctured cubes in which we've made the above identifications
\begin{center}
    \begin{tikzcd}[column sep = 5pt, row sep=5pt]
                        &                         & 0 \arrow[rd] \arrow[dd] &                     &  &                         &                         & 0 \arrow[dd] \arrow[rd] &                     \\
                        & 0 \arrow[rr, crossing over] \arrow[dd, crossing over] &                         & A \arrow[dd, "m_A"] &  &                         & 0 \arrow[dd, crossing over] \arrow[rr, crossing over] &                         & B \arrow[dd, "m_B"] \\
0 \arrow[rd] \arrow[rr] &                         & A \arrow[rd, "i_1"]     &                     &  & 0 \arrow[rd] \arrow[rr] &                         & B \arrow[rd, "i_1"]     &                     \\
                        & A \arrow[rr, "i_2"']    &                         & A\vee A             &  &                         & B \arrow[rr, "i_2"']    &                         & B\vee B            
\end{tikzcd}
\end{center}

The data of this morphism consists of three maps $f_1, f_2, f_3:A\to B$ as well as a map $f_4:A\vee A\to B\vee B$ which makes the diagram formed by these maps and the two cubes commute.  Because the bottom squares of each cube are pushouts we have $f_4 \simeq f_1\vee f_2$
\begin{center}
    \begin{tikzcd}
0 \arrow[d] \arrow[r]                                 & A \arrow[d, "i_2"] \arrow[rdd, "i_2f_1", bend left] &         \\
A \arrow[r, "i_1"] \arrow[rrd, "i_1f_2"', bend right] & A\vee A \arrow[rd, "f_4" description]               &         \\
                                                      &                                                     & B\vee B
\end{tikzcd}
\end{center}
By one half of counitality we obtain an equivalence $f_3 \simeq f_1$ as shown below
\begin{center}
\begin{tikzcd}
A \arrow[r, "f_3"] \arrow[d, "m_A"] \arrow[dd, no head, equals, bend right = 1.7cm] & B \arrow[d, "m_B"'] \arrow[dd, no head, equals, bend left = 1.7cm] &   & A \arrow[r, "f_3"] \arrow[d, no head, equals] & B \arrow[d, equals] \\
A\vee A \arrow[r, "f_1\vee f_3"] \arrow[d]                                      & B\vee B \arrow[d]                                              & = & A \arrow[r, "f_1"]                                & B                       \\
A \arrow[r, "f_1"]                                                              & B                                                              &   &                                                   &                        
\end{tikzcd}
\end{center}
Likewise we have $f_3 \simeq f_2$.  Let $f$ denote any map equivalent to one of $f_1, f_2, f_3$.  The morphism of punctured cubes gives a morphism of $A_2$-cogroup objects
\begin{center}
    \begin{tikzcd}
0 \arrow[r] & A \arrow[r, shift right] \arrow[r] \arrow[r, shift left] \arrow[d, "f" ] & A\vee A \arrow[d, "f\vee f"] \\
0 \arrow[r] & B \arrow[r, shift right] \arrow[r] \arrow[r, shift left]                            & B\vee B                                 
\end{tikzcd}
\end{center}

Similarily we can show in general that a $k$-simplex in $T_n\C$ gives rise to a $k$-simplex in $coGrp_2(\C)$ in an essentially unique way, so that there is an equivalence $T_2\C\simeq coGrp_2(\C)$.
\end{ex}

We will prove Theorem \ref{cubeiscogroup} using induction by considering the process of building an $A_n$-cogroup object out of an $A_{n-1}$-cogroup object.  Lurie tackles this question in \cite[Section 4.1]{HA} by introducing an intermediate object, an $A_n$-precogroup object, which adds no additional information to an $A_{n-1}$-cogroup object.

\begin{defn}
Let $\Delta_n^0$ be the largest simplicial subset of $\Delta_n$ which does not include any simplex which includes the edge $e:[1]\to [n]$ as a face.
\end{defn}

\begin{defn}
 An \textit{$A_n$-precogroup object} of $\C$ is a functor $A^0:\Delta_n^0\to \C$ which satisfies the conditions  $(1)$ and $(2')$ of Definition \ref{cogroup}, when they make sense, and additionally satisfies condition $(3)$ if $n > 2$.  We denote the $\infty$-category of $A_n$-precogroup objects in $\C$ by $coGrp_n^{pre}(\C)$.
\end{defn}

\begin{thm*}\cite[Theorem 4.1.5.8]{HA}
The inclusion $\Delta_{n-1}\hookrightarrow \Delta^0_n$ induces a trivial fibration 
\[coGrp_n^{pre}(\C)\overset{\sim}{\longrightarrow} coGrp_{n-1}(\C)  \]
\end{thm*} 

We will utilize similar objects defined as diagrams on a cubical version of $\Delta^0_n$.
\begin{defn}
Let $\square[n+1]_0^0$ be defined by the pullback of simplicial sets
\begin{center}
    \begin{tikzcd}
        \square[n+1]_0^0\ar[r]\ar[d, hook] & \Delta^0_n\ar[d, hook] \\
        \square[n+1]_0\ar[r] & \Delta_n
    \end{tikzcd}
\end{center}
\end{defn}

\begin{lem}\label{pushout} The pullback square \begin{center}
    \begin{tikzcd}
        \square[n+1]_0^0\ar[r]\ar[d, hook] & \Delta^0_n\ar[d, hook] \\
        \square[n+1]_0\ar[r] & \Delta_n
    \end{tikzcd}
\end{center}
is additionally a pushout square.
\end{lem}
\begin{proof}
Consider a chain of inclusions $[m_0]\xrightarrow{f_1}[m_1]\xrightarrow{f_2}\dots\xrightarrow{f_k}[m_k]$ in $\Delta_n$ which does not belong to $\Delta_n^0$. We want to show that such a $k$-simplex lifts to one in $\square[n+1]_0$. Necessarily $[m_k] = [n]$.  The lift is given by the image of the chain of inclusions $J_0 \subseteq J_1 \subseteq \dots\subseteq [n]$ where $J_i = im(f_k\circ\cdots\circ f_{i+1})$.  This is similar to the proof of \cite[Proposition 4.1.5.7]{HA}.
\end{proof}

There is an inclusion \[P(\{0,\dots, n-1\})\cup P_0(\{1,\dots, n\})\hookrightarrow P_0(\{0,\dots, n\})\] by partitioning $P_0(\{0,\dots, n\})$ into those subsets which respectively contain and do not contain 0. By pulling back this gives a functor \[p:Fun(\square[n+1]_0, \C)\to Fun(\square[n], \C)\times Fun(\square[n]_0, \C) \] 

\begin{defn}
Let $Fun^{sco}(\square[n]_0, \C)$ denote the full subcategory of $Fun(\square[n]_0, \C)$ spanned by the $(n)$-cubes which are strongly cocartesian and which have initial vertex a 0 object. We call these objects \textit{special cocartesian $(n)$-cubes}.
\end{defn}

The functor $p$ restricts to a functor on special punctured cubes \[ T_n(\C)\to Fun^{sco}(\square[n], \C)\times T_{n-1}(\C) \] 
In this setup we view a special punctured $(n+1)$-cube is containing the data of a special punctured $n$-cube, a special cocartesian $n$-cube, and a collection of ``connecting structure" between the two.

The functor $p$ factors through a functor \[Fun(\square[n+1]_0^0, \C)\to Fun(\square[n], \C)\times Fun(\square[n]_0, \C)\]  which motivates a definition of special $\square[n+1]_0^0$ diagrams.

\begin{defn}
We define a category $T_n^0(\C)$ of special $\square[n+1]_0^0$ diagrams in an $\infty$-category $\C$ as the pullback
\begin{center}
    \begin{tikzcd}
    T_n^0(\C)\ar[r]\ar[d] & Fun(\square[n]_0^0, \C)\ar[d] \\
    Fun^{sco}(\square[n], \C)\times T_{n-1}(\C)\ar[r] & Fun(\square[n], \C)\times Fun(\square[n]_0, \C)
    \end{tikzcd}
\end{center}
\end{defn}
Note that the forgetful functor \[Fun(\square[n+1]_0, \C)\to Fun(\square[n+1]_0^0, \C)\] restricts to a functor $T_n(\C)\to T_n^0(\C)$.

Again we can view an object of $T_n^0(\C)$ as special punctured $n$-cube, a special cocartesian $n$-cube, and a smaller collection of connecting structure between the two. A key idea will be that the special cocartesian $n$-cubes coming from a special punctured $(n+1)$-cube are already determined by the rest of the structure.  Already the $\infty$-category $Fun^{sco}(\square[n], \C)$ contains very little information.

\begin{lem}\label{sconodata}
The functor $Fun^{sco}(\square[n], \C)\xrightarrow{\sim} \C^n$ given by pulling back along the inclusion of cardinality one subsets of $[n]$ is an equivalence.
\end{lem}
\begin{proof}
There is a fully faithful left Kan extension functor \[Fun(P_{\le1}(\{0,\dots,n-1\}, \C)\hookrightarrow Fun(\square[n], \C)\] from diagrams defined on subsets of $[n-1]$ of cardinality less than two to diagrams defined on the entire powerset of $[n-1]$.  $Fun^{sco}(\square[n], \C)$ can be identified with the image of this functor on the subcategory of $Fun(P_{\le1}(\{0,\dots,n-1\}, \C)$ of diagrams with initial vertex a 0 object.  This category is equivalent to $\C^n$.
\end{proof}

In Example \ref{case3} we saw by taking cofibers that a special punctured 3-cube takes all initial non-zero vertices to equivalent objects.  We will need a generalization of this to special punctured $n$-cubes.

\begin{lem}\label{clawlemma}
A special punctured $(n+1)$-cube $\cX:\square[n+1]_0\to \C$ takes all cardinality two subsets of $[n]$ to equivalent objects in $\C$.
\end{lem}
\begin{proof}
The case for $n=2$ was proved already in Example \ref{case3}.  If $n>2$ we can choose four distinct elements $\{a,b,c,d\}$.  Consider the subsets $\{a,b\}, \{b,c\}$ and form the special punctured 3-cube 
\begin{center}
    \begin{tikzcd}[column sep = 5pt, row sep=5pt]
                            &                             & \{b\} \arrow[dd] \arrow[rd] &                      \\
                            & \{c\} \arrow[dd] \arrow[rr] &                             & {\{b,c\}} \arrow[dd] \\
\{a\} \arrow[rr] \arrow[rd] &                             & {\{a,b\}} \arrow[rd]        &                      \\
                            & {\{a,c\}} \arrow[rr]        &                             & {\{a,b,c\}}         
\end{tikzcd}
\end{center}
Any special punctured $n$-cube $\cX$ takes this subdiagram to a special punctured 3-cube, providing an equivalence $\cX\{a,c\}\simeq \cX\{b,c\}$ via the argument of Example \ref{case3}.  Likewise we can form a punctured 3-cube amongst $\{b,c,d\}$ and obtain an equivalence $\cX\{b,c\}\simeq \cX\{b,d\}$, giving an equivalence $\cX\{a,c\}\simeq\cX\{b,d\}$.  In this way we can connect any two cardinality two subsets of $[n]$ via an equivalence.
\end{proof}

We are now ready to prove our first main result.

\begin{thm}\label{cubeiscogroup}
Let $\C$ be a pointed $\infty$-category with finite colimits. The functor $\pi_n^*:coGrp_n(\C)\rightarrow T_n(\C)$ is an equivalence of $\infty$-categories
\end{thm}
\begin{proof}
This is clearly the case for $n=1$.  We will show that the square 
\begin{center}
    \begin{tikzcd}
coGrp_n(\C) \arrow[r, "\pi_n"] \arrow[d] &   T_n(\C) \arrow[d] \\
coGrp_{n-1}(\C) \arrow[r, "\pi_{n-1}"]        & T_{n-1}(\C)      
\end{tikzcd}
\end{center}
is a pullback, proving the claim by induction.  This will be accomplished by showing that for any $A_{n-1}$-cogroup object $A$ in $\C$ there is an equivalence of fibers \[ coGrp_{n}(\C)_A \simeq T_n(\C)_{\pi_{n-1}A} \] 

By \cite[Theorem 4.1.5.8]{HA} there is an essentially unique extension of $A$ to an $A_n$-precogroup object $A^0:\Delta^0_n\to \C$  and so an equivalence of fibers $coGrp_n(\C)_A\simeq coGrp_n(\C)_{A^0}$.  From the pushout square of Lemma \ref{pushout} we obtain a pullback square
\begin{center}
\begin{tikzcd}
    coGrp_n(\C)\ar[r, "\pi_n"]\ar[d] & T_n(\C)\ar[d] \\
    coGrp_n^{pre}(\C)\ar[r, "\pi^0_n"] & T^0_n(\C)
    
\end{tikzcd}
\end{center}
and hence an equivalence of fibers $coGrp_n(\C)_{A^0} \simeq T_n(\C)_{\pi^0_nA^0}$.  It suffices to show now that the natural map of fibers $T_n(\C)_{\pi^0_nA^0}\to T_n(\C)_{\pi_{n-1} A}$ is an equivalence.

By abuse of notation, let $A$ denote the object $A[1]$, The functor $\pi_n^0A^0$ factors via the pullback in the diagram below since the cardinality two vertices are all equal to $A$. Note also that the pullback in the bottom left vertex of the diagram below is a homotopy pullback, so that the square is an actual pullback.
\begin{center}
    \begin{tikzcd}[column sep=0.3cm]
{[0]} \arrow[rr] \arrow[rrrr, "\pi_n^0A^0", bend left=10] &  & \bullet \arrow[rr] \arrow[d]                                                  &  & {Fun(\square[n+1]_0^0, \C)} \arrow[d]                 \\
                                                        &  & {(Fun^{sco}(\square[n], \C)\times T_{n-1}(\C))\times_{C^{2n}} \{A\}} \arrow[rr] &  & {Fun(\square[n], \C)\times Fun(\square[n]_0, \C)}
\end{tikzcd}
\end{center}
The weak equivalence $Fun^{sco}(\square[n], \C)\xrightarrow{\sim} \C^n$ of Lemma \ref{sconodata} gives a weak equivalence \[(Fun^{sco}(\square[n], \C)\times T_{n-1}(\C))\times_{\C^{2n}} \{A\} \xrightarrow{\sim} T_{n-1}(\C)\times_{\C^n} \{A\}\] which allows us to greatly simplify the bottom left vertex of the above diagram.

Note now that  $\pi_n^0 A^0$ lands in $T_n^0(\C)\hookrightarrow Fun(\square[n+1]_0^0, \C)$ and that under the above equivalence $T_{n-1}(\C)\times_{\C^n} \{A\}$ maps into \[T_{n-1}(\C)\times \C^n\simeq Fun^{sco}(\square[n], \C)\times T_{n-1}(\C)\ \hookrightarrow Fun(\square[n], \C)\times Fun(\square[n]_0, \C)\] 
So that the above square factors and we can compute the fiber $T_n(\C)_{\pi_n^0A^0}$ equally as the pullback in the left hand square below, which is seen to be a pullback using two applications of the pullback pasting lemma.
\begin{center}
\begin{tikzcd}[column sep=0.3cm, row sep =0.3cm]
T_n(\C)_{\pi_n^0A^0} \arrow[rr] \arrow[d] &  & P^0 \arrow[d] \arrow[rr]                                                        &  & T_n(\C) \arrow[d]                             \\
                                           {[0]} \arrow[rr]  &  & \bullet \arrow[d] \arrow[rr]                                                    &  & T_n^0(\C) \arrow[d]                           \\
                         &  & T_{n-1}(\C) \times_{\C^n}\{A\} \arrow[rr] &  & { T_{n-1}(\C)\times \C^n}
\end{tikzcd}
\end{center}
Likewise $\pi_{n-1}A$ factors via $T_{n-1}(\C)\times_{\C^n}\{A\}$ and we can compute the fiber $T_n(\C)_{\pi_{n-1}A}$ as the left hand pullback in
\begin{center}
\begin{tikzcd}
T_n(\C)_{\pi_{n-1}A} \arrow[rr] \arrow[d] &  & P \arrow[rr] \arrow[d]                   &  & T_{n}(\C) \arrow[d] \\
{[0]} \arrow[rr]                          &  & T_{n-1}(\C)\times_{\C^n}\{A\} \arrow[rr] &  & T_{n-1}(\C)        
\end{tikzcd}
\end{center}
Now compare the above right square with the square in the definition of $P^0$.
\begin{center}
\begin{tikzcd}
P^0 \arrow[rr] \arrow[d]                   &  & T_{n}(\C) \arrow[d]    \\
T_{n-1}(\C)\times_{\C^n}\{A\} \arrow[rr] &  & T_{n-1}(\C)\times \C^n
\end{tikzcd}
\end{center}
We consider the difference in these two squares. By construction we demand that the special punctured $(n+1)$-cubes in $P^0$ and $P$ both are equal to $A$ on every initial vertex in the subset $\square[n]_0\hookrightarrow \square[n+1]_0$.  In addition, in $P^0$ we demand that the vertices of cardinality one in $\square[n]\hookrightarrow \square[n+1]_0$ are also equal to $A$.  By Lemma \ref{clawlemma} this later condition is automatic in a special punctured $(n+1)$-cube already: the special $n$-cube in a special punctured $(n+1)$-cube provides no additional data. Because of this we have an equivalence $P\simeq P^0$, and so finally we obtain an equivalence of fibers $T_n(\C)_{\pi_n^0A^0}\simeq T_n(\C)_{\pi_{n-1}A}$ as desired.
\end{proof}

\section{coGroups in Goodwillie Calculus}\label{goodwillie}
We begin by recalling the parts of the theory of Goodwillie approximation to $\infty$-categories to which we will apply Theorem \ref{cubeiscogroup}. We work always in the category $\cat^\omega$ of pointed, compactly generated $\infty$-categories with morphisms the left adjoint functors which preserve compact objects.  Note that by the adjoint functor theorem such a functor necessarily admits a right adjoint.  For more details on this category see \cite[Appendix A]{H}.

\begin{defn} (Heuts)
An adjunction $F:\C\rightleftarrows \cD:G$ is a \textit{weak $n$-excisive approximation} to $\C$ if
\begin{itemize}
    \item The identity $1_{\cD}$ is $n$-exsisive.
    \item The induced maps $P_n(1_\C)\xrightarrow{\sim} GF$ and $P_n(FG)\xrightarrow{\sim} 1_{\cD}$ are equivalences.
\end{itemize}
\end{defn}

\begin{defn} (Heuts)
An $\infty$-category $\cD$ is \textit{$n$-excisive} if every weak $n$-excisive approximation to $\cD$ is an equivalence. An adjunction $F:\C\rightleftarrows \cD:G$ is an $n$-excisive approximation to $\C$ if it is a weak $n$-excisive approximation and the $\infty$-category $\cD$ is $n$-excisive.
\end{defn}

\begin{ex}
An $\infty$-category is $1$-excisive if and only if it is stable. The $1$-excisive approximation to $\cS_\star$, the $\infty$-category of pointed spaces, is the usual stabilization adjunction $\Sigma^\infty:\cS_\star\rightleftarrows Sp:\Omega^\infty$.
\end{ex}

Just as one can construct $n$-excisive approximations to functors, there is a natural construction of an $n$-excisive approximation to an $\infty$-category using the category $T_n\C$.

Define the $\infty$-category $\overline{T_n}\C$ of (nonpunctured) $(n+1)$-cubes in $\C$, whose objects consist of those cubes with a zero object at the one element subsets and such that every face is a pushout.  Such a cube is completely determined by its initial vertex: it is a right Kan extension for the initial vertex to the cardinality one subsets followed by a left Kan extension to the entire cube. There is a natural span of functors
\begin{center}
    \begin{tikzcd}
        &   \overline{T_n}\C\ar[dl, two heads,"\sim"']\ar[dr]    &   \\
        \C      &&  T_n\C   
    \end{tikzcd}
\end{center}
with left leg, evaluation at the initial vertex, a trivial Kan fibration by \cite[Thm 4.3.2.15]{HTT}, and with right leg the functor that forgets the initial vertex.  Choosing a section to the trivial fibration gives a functor $L_n:\C\rightarrow T_n\C$ which sends an object of $\C$ to the punctured $(n+1)$-cube which successively takes pushouts of suspensions of the object, and then forgets the initial vertex. Note that the functor $L_n$ admits a right adjoint $R_n:T_n\C\rightarrow \C$ which takes the limit of the punctured cube.

\begin{ex}\label{example}
Let $A\in\C$.  The functor $L_2:\C\rightarrow T_n\C$ takes A to a cube with vertices equivalent to the objects below left, and then forgets the initial vertex giving the cube below right
\begin{center}
    \begin{tikzcd}[row sep = 5pt, column sep = 5pt]
    A\ar[rr]\ar[dd]\ar[dr]\   &&  0\ar[dd]\ar[dr]   &       \\
        &   0\ar[dd]\ar[rr]   &&  \Sigma A\ar[dd]   \\
    0\ar[rr]\ar[dr]   &&  \Sigma A\ar[dr]   &       \\
        &   \Sigma A\ar[rr]   &&  \Sigma A \vee \Sigma A   
    
    \end{tikzcd}
    \hspace{5pt}
    \begin{tikzcd}[row sep = 5pt, column sep = 5pt]
       &&  0\ar[dd]\ar[dr]   &       \\
        &   0\ar[dd]\ar[rr]   &&  \Sigma A\ar[dd]   \\
    0\ar[rr]\ar[dr]   &&  \Sigma A\ar[dr]   &       \\
        &   \Sigma A\ar[rr]   &&  \Sigma A \vee \Sigma A   
    \end{tikzcd}
\end{center}
Note that identifying the bottom square with the coproduct inclusion of $\Sigma A$ into $\Sigma A \vee \Sigma A$ we recover the $A_2$-cogroup structure on $\Sigma A$, as in Example \ref{case3}.
\end{ex}


Using the functor $L_n:\C\rightarrow T_n\C$ one may build a chain 
\[\C\rightarrow  T_n\C \rightarrow T_n(T_n\C)\rightarrow T_n(T_n(T_n\C))\rightarrow\dots P_n\C = \colim_k T_n^k\C \]
and an adjunction $\Sigma^\infty_n\C\leftrightarrows P_n\C: \Omega^\infty_n$.

\begin{prop*}\cite[Proposition 3.15]{H}
Let $\C$ be a pointed, compactly generated $\infty$-category. Then the adjunction $\Sigma^\infty_n\C\rightleftarrows P_n\C: \Omega^\infty_n$ resulting from the above construction is an $n$-excisive approximation to $\C$.
\end{prop*} 
Using the forgetful functors $T_n(\C)\to T_{n-1}(\C)$ we can construct a tower of compact object preserving left adjoint functors \[\dots\xrightarrow{\Sigma_{n+2, n+1}}P_{n+1}\C\xrightarrow{\Sigma_{n+1, n+}}P_n\C\xrightarrow{\Sigma_{n, n-1}}\dots P_1\C \simeq Sp(\C)\]

In fact $n$-excisiveness is detected already at the level of $T_n\C$.

\begin{prop*}\cite[Proposition 3.16]{H}\label{approx}
Let $\C$ be a pointed, compactly generated $\infty$-category.  Then $\C$ is $n$-excisive if and only if the functor $L_n:\C\rightarrow T_n\C$ is an equivalence.
\end{prop*}

This characterization of $n$-excisive $\infty$-categories \cite[Proposition 3.16]{H} combined with the equivalence $coGrp_n(\C)\simeq T_n(\C)$ immediately give our second main result.

\begin{thm}\label{layer}
Let $\C$ be a pointed, compactly generated $\infty$-category. Then $\C$ is $n$-excisive if and only if the  suspension functor $\Sigma_\C$ induces an equivalence
\[\Sigma:\C\xrightarrow{\sim} coGrp_n(\C) \] between $\C$ and the $\infty$-category of $A_n$-cogroup objects in $\C$.
\end{thm}

\begin{cor}
Let $\C$ be a pointed, compactly generated $\infty$-category, then the suspension functor on $P_n\C$ induces an equivalence
\[P_n\C\xrightarrow{\sim} coGrp_n(P_n\C) \]
\end{cor}

Theorem \ref{layer} has a nice conceptual interpretation.  Of course in a stable $\infty$-category all objects may be desuspended, hence desuspension follows from the trivial $A_1$-cogroup structure on an object.  In the unstable setting, dual to the recognition theorem for loop spaces, one would assume a full $A_\infty$-cogroup structure is needed for a desuspension\footnote{We take on this idea in the case of spaces in Section \ref{spaces}.}.  What we have proved here is that in the intermediate $n$-excisive setting an $A_n$-cogroup structure suffices, and in fact this characterizes exactly the $n$-excisive setting.

It is interesting to reflect for a moment on the process of stabilization in a new light. When one stabilizes an $\infty$-category $\C$ they invert the suspension functor $\C\xrightarrow{\Sigma}\C$ by forming the colimit \[ \C\xrightarrow{\Sigma} \C\xrightarrow{\Sigma}\C\xrightarrow{\Sigma}\cdots \rightarrow P_1\C\]
This process forgets a lot of information about the values of the suspension functor, in particular the $A_\infty$-cogroup structure that they carry. We can carry out a less harshly forgetful inversion by remembering a part of this structure, say the $A_n$-cogroup structure, and then inverting the suspension functor by forming the colimit
\[ \C\xrightarrow{\Sigma} coGrp_n(\C)\xrightarrow{\Sigma}coGrp_n(coGrp_n(\C))\xrightarrow{\Sigma}\cdots \rightarrow P_n\C\]
Theorem \ref{layer} says that forming the $n$-excisive approximation $P_n\C$ is exactly this procedure.  We can also give the following slight weakening

\begin{prop}
Let $\C$ be a pointed, compactly generated $\infty$-category, then $1_\C$ is $n$-excisive if and only if the suspension functor $\Sigma: \C\rightarrow coGrp_n(\C)$ is fully faithful.
\end{prop}
\begin{proof}
This follows from Theorem \ref{layer} since $P_n\C$ is a filtered colimit of fully faithful functors $\Sigma$ starting from $\C$.
\end{proof}

Our goal now is to push Theorem \ref{layer} to the limit and study $A_\infty$-cogroup objects in the limit of the tower of excisive approximations. To begin we will state some technical results in the following common setup.

\begin{notation}\label{setup} Consider a chain of left adjoint functors of compactly generated $\infty$-categories \[ \dots\xrightarrow{f_{n+2}} \C_{n+1}\xrightarrow{f_{n+1}} \C_n\xrightarrow{f_n} \dots\rightarrow C_1 \] Let $\C$ denote the limit, and $\pi_n:\C\rightarrow \C_n$ the $n^{th}$ projection functor.
\end{notation}

\begin{defn}
We say a family of functors $f_\alpha:C\to C_\alpha$ jointly creates colimits (or equivalences) if the induced functor $f:C\to \prod_\alpha C_\alpha$ creates colimits (resp. equivalences).
\end{defn}

\begin{lem}\label{jointcolimits}
In the setup of Notation \ref{setup} the functors $\C\xrightarrow{\pi_n}\C_n$ jointly create equivalences and colimits.
\end{lem}
\begin{proof}
This follows from \cite[Theorem 1.1]{RVC} or as explained in the (not static!) discussion at the end of the proof of \cite[Proposition 8.2.11]{RVE}.
\end{proof}

\begin{prop}\label{createcomonoid}
In the setup of Notation \ref{limit}, for any $n$ there are equivalences of $\infty$-categories
\[ coMon_n(C)\xrightarrow{\sim} lim_k coMon_n(C_k) \]
\end{prop}
\begin{proof}
There is an equivalence  
\[Fun(\Delta_n, \C)\xrightarrow{\sim} lim_k Fun(\Delta_n, \C_k) \]
which we must show lifts to an equivalence 
\[coMon_n(\C)\xrightarrow{\sim} lim_k coMon_n(C_k) \]
By Lemma \ref{jointcolimits} the legs of the limit cone $\C\rightarrow \C_k$ create colimits and equivalences, so they create the Segal condition (2') of Definition \ref{cogroup}.
\end{proof}

Now that we understand how to form $A_n$-cogroup objects in the limit of a tower of $\infty$-categories each fixed $n$, we deal with the limit as $n$ goes to infinity.

\begin{lem}\label{filterinfinity}
Let $\C$ be an $\infty$-category. The sequence of forgetful functors 
\[ coMon_{\infty}(\C)\rightarrow\dots \rightarrow coMon_2(\C)\rightarrow coMon_1(\C) \]
induces an equivalence \[  coMon_{\infty}(\C)\overset{\sim}{\longrightarrow} lim_k coMon_k(\C) \]
\end{lem}
\begin{proof}
This is dual to the unital version of \cite[Prop 4.1.4.9]{HA}.
\end{proof}

Because being cogroup-like can be tested already at the level of the homotopy category, combining Proposition \ref{createcomonoid} and Lemma \ref{filterinfinity} we obtain:

\begin{prop}\label{comonoidsinlimit}
In the setup of Notation \ref{setup} there is an equivalence of $\infty$-categories
\[ coMon_{\infty}(\C)\xrightarrow{\sim} lim_{k,n} coMon_k(\C_n) \] which decends to an equivalence \[ coGrp_{\infty}(\C)\xrightarrow{\sim} lim_{k,n} coGrp_k(\C_n) \] on the level of cogroup objects.
\end{prop}

We will now apply Proposition \ref{comonoidsinlimit} to the special case of Goodwillie calculus. Recall that there is a sequence of compact object preserving left adjoint functors
\[\dots\xrightarrow{\Sigma_{n+2, n+1}}P_{n+1}\C\xrightarrow{\Sigma_{n+1, n+}}P_n\C\xrightarrow{\Sigma_{n, n-1}}\dots P_1\C \simeq Sp(\C)\]

We let $P_\infty \C$ denote the limit of this tower in the category $\cat^\omega$, which is complete by \cite[Appendix A]{H}.  In fact one can compute this limit by passing to the sequence \[\dots\xrightarrow{\Sigma_{n+2, n+1}}(P_{n+1}\C)^c\xrightarrow{\Sigma_{n+1, n+}}(P_n\C)^c\xrightarrow{\Sigma_{n, n-1}}\dots (P_1\C)^c \simeq Sp(\C^c)\]
computing the limit in $\cat$, and taking Ind-objects of the resulting category.

\begin{thm}\label{limit}
Let $\C$ be a pointed, compactly generated $\infty$-category. The suspension functor \[\Sigma: P_\infty\C\overset{\sim}{\longrightarrow} coGrp_\infty(P_\infty\C)\] is an equivalence of $\infty$-categories.
\end{thm}
\begin{proof}
The equivalences $P_n(\C)\xrightarrow{\Sigma} coGrp_n(P_n\C)$ assemble into a tower of equivalences

\begin{center}
\begin{tikzcd}
P_{\infty}\C\ar[d, dotted]\ar[r, dotted]      &   lim_n coGrp_n(P_n\C)\ar[d, dotted] \\
P_{n+1}\C\ar[d]\ar[r, "\sim"', "\Sigma"]   &   coGrp_{n+1}(P_{n+1}\C)\ar[d] \\
P_{n}\C\ar[d]\ar[r, "\sim"', "\Sigma"]     &   coGrp_{n}(P_{n}\C)\ar[d] \\
P_{1}\C\ar[r, "\sim"', "\Sigma"]           &   coGrp_{1}(P_{1}\C) 
\end{tikzcd}
\end{center}

Providing an equivalence $P_{\infty} \xrightarrow{\sim} lim_n coGrp_n(P_n\C)$.  Because the functors $P_{\infty}\C\rightarrow P_n\C$ jointly create colimits we may identify this functor on underlying objects with the suspension $\Sigma: P_\infty\C\rightarrow P_{\infty}\C$.

On the other hand, the family of functors $P_{\infty}\C\xrightarrow{\pi_n}P_n\C$ satisfy the setup of Notation \ref{setup}, and so Proposition \ref{comonoidsinlimit} provides an equivalence  \[ coGrp_\infty(P_{\infty}\C) \xrightarrow{\sim} lim_{k,n}coGrp_k(P_n\C)\]
However, the diagonal functor $\Delta: \mathbb{N}\rightarrow \mathbb{N}\times\mathbb{N}$ is final, inducing an equivalence \[ lim_ncoGrp_n(P_n\C)\xrightarrow{\sim} lim_{k,n}coGrp_k(P_n\C) \] Providing the desired equivalence $\Sigma: P_{\infty}\C\xrightarrow{\sim} coGrp_{\infty}(P_{\infty}\C)$ which finishes the proof of Theorem \ref{limit}.
\end{proof}

As in the discussion following Theorem \ref{layer}, one could attempt to invert the suspension with its full structure 
\[ \C\xrightarrow{\Sigma} coGrp_\infty(\C)\xrightarrow{\Sigma}coGrp_\infty(coGrp_\infty(\C))\xrightarrow{\Sigma}\cdots \rightarrow \hat{P}_\infty\C\]
defining an $\infty$-category $\hat{P}_\infty\C$ as the colimit. This begs the question question: are there equivalences $\C\simeq \hat{P}_\infty\C$ or $\hat{P}_\infty\C\simeq P_\infty\C$?
Unfortunately neither of these seem to be true.  There is a natural map $\hat{P}_\infty\C \rightarrow P_\infty\C$, which can be obtained for example by contemplating the following $\mathbb{N}\times \mathbb{N}^{op}$ indexed diagram $\cJ_{\bullet, \bullet}$.
\begin{center}
\begin{tikzcd}[row sep = 10pt, column sep=7pt]
\vdots \arrow[d,equals]       &  & \vdots \arrow[d]                 &  & \vdots \arrow[d]                         &        \\
C \arrow[d, equals] \arrow[rr] &  & coGrp_3(\C) \arrow[rr] \arrow[d] &  & coGrp_3(coGrp_3(\C)) \arrow[r] \arrow[d] & \cdots \\
C \arrow[d, equals] \arrow[rr] &  & coGrp_2(\C) \arrow[rr] \arrow[d] &  & coGrp_2(coGrp_2(\C)) \arrow[r] \arrow[d] & \cdots \\
C \arrow[rr]           &  & coGrp_1(\C) \arrow[rr]           &  & coGrp_1(coGrp_1(\C)) \arrow[r]           & \cdots
\end{tikzcd}
\end{center}

\begin{lem}
In the diagram $\cJ_{\bullet, \bullet}$ above there are natural equivalences
\[ \hat{P}_\infty\C \simeq colim_n lim_m \cJ_{n,m} \text{       and       }  P_\infty\C \simeq lim_n colim_m \cJ_{n,m} \]
In particular the natural functor  $colim_n lim_m \cJ_{n,m}\rightarrow lim_n colim_m \cJ_{n,m}$ induces a functor $\hat{P}_\infty\C \rightarrow P_\infty\C$ which factors the comparison functor $\C\rightarrow P_\infty\C$.
\end{lem}
\begin{proof}
This follows from Lemma \ref{filterinfinity} and the proof of Theorem \ref{limit}.
\end{proof}

Because filtered colimits and filtered limits do not commute in general in $\cat^\omega$ it seems unlikely that this map is an equivalence.  However one does have this result in the case that the Goodwillie tower converges.

\begin{lem}
Suppose the comparison functor $\C\rightarrow P_\infty(\C)$ is an equivalence, then the functor  $\hat{P}_\infty\C \rightarrow P_\infty\C$ is an equivalence.
\end{lem}
\begin{proof}
If $\C\xrightarrow{\sim} P_\infty(\C)$ is an equivalence then $\Sigma:\C\xrightarrow{\sim} coGrp_\infty(\C)$ is an equivalence by Theorem \ref{limit} and so $\C\xrightarrow{\sim}\hat{P}_\infty\C$ is an equivalence.  The result then follows by the two out of three property for equivalences.
\end{proof}
It does not however appear that $\C\rightarrow \hat{P}_\infty\C$ being an equivalence implies that $\C\rightarrow P_\infty\C$ is an equivalence. Convergence of the Goodwillie tower remains a stronger result.

\section{Desuspensions in Spaces}\label{spaces}
In this section we study how our results interact with the convergence of the Goodwillie tower of spaces.  Heuts shows that the layers of the tower are equivalent to $\infty$-categories of truncated ``Tate coalgebras" in spectra. For a precise development see \cite{H}. Let $P_\infty \cS_*^{\ge n}$ denote the full sub $\infty$-category of $P_\infty \cS_*$ whose associated Tate coalgebra spectra are n-connected.  Heuts proves the following convergence result of the Goodwillie tower of spaces.

\begin{thm*}\cite[Theorem 1.3]{H}\label{convergencespacesgijs}
There is an equivalence of categories for $n > 1$ \[\cS_*^{\ge n} \overset{\sim}{\longrightarrow} P_\infty \cS_*^{\ge n} \]
\end{thm*}

Theorem \ref{limit} in this case gives us the following result.

\begin{cor}\label{limitspaces}
For any $n>0$ the suspension functor induces an equivalence of $\infty$-categories 
\[ \Sigma: P_\infty \cS_*^{\ge n} \overset{\sim}{\longrightarrow} coGrp_{\infty}(P_\infty \cS_*^{\ge n+1})\]
\end{cor}

Note that a 2-connected $A_{\infty}$ comonoid is automatically cogroup-like. Combining this observation, Corollary \ref{limitspaces}, and \cite[Theorem 1.3]{H} we recover a theorem of Klein-Schwänzl-Vogt, \cite[Theorem 1.4]{KSV} in the form of an equivalence of $\infty$-categories.

\begin{thm}\label{ksv}
The suspension functor induces an equivalence of $\infty$-categories
\[ \Sigma: \cS_*^{\ge 2} \overset{\sim}{\longrightarrow} coMon_{\infty}(\cS_*^{\ge 3}) \]
In particular every 2-connected $A_{\infty}$-comonoid admits a desuspension to a 1-connected space.
\end{thm}
\begin{proof}
Consider the square
\begin{center}
    \begin{tikzcd}
    S_*^{\ge 2}\ar[r, "\sim"]\ar[d]     &       P_{\infty}S_\star^{\ge 2}\ar[d, "\sim"] \\
    coGrp_{\infty}(S_\star^{\ge 3})\ar[r, "\sim"] &  coGrp_{\infty}(P_{\infty}S_\star^{\ge 3})
    \end{tikzcd}
\end{center}
By \cite[Theorem 1.3]{H} the horizontal functors are equivalences and by Corollary \ref{limitspaces} the right vertical map is an equivalence, hence also the left hand vertical functor is an equivalence.
\end{proof}

The benefit to our approach has a distinct advantage to that of \cite{KSV}, where the authors prove their claims through a variety of arguments involving connectivity.  Our approach isolates these technicalities to a question purely about the convergence of the Goodwillie tower. This general framework also allows us to attack the question of desuspension in a general pointed, compactly generated $\infty$-category.

In the 1-connected case, and in fact the Goddwillie convergent case which includes nilpotent spaces, we can only make a weaker claim of fully faithfulness. Let $\cS_*^{conv}$ denote the $\infty$-category of connected spaces for which the Goodwillie tower of the identity converges.

\begin{prop}
The suspension functor induces a fully faithful functor of $\infty$-categories
\[ \Sigma: \cS_*^{conv } \hookrightarrow coGrp_{\infty}(\cS_*^{\ge 2}) \]
\end{prop}
\begin{proof}
In the corresponding square for the case of $S_*^{conv}$
\begin{center}
    \begin{tikzcd}
    S_*^{conv}\ar[r, hook]\ar[d, hook]     &       P_{\infty}\cS_\star^{\ge 1}\ar[d, "\sim"] \\
    coGrp_{\infty}(S_\star^{\ge 2})\ar[r, "\sim"] &  coGrp_{\infty}(P_{\infty}S_\star^{\ge 2})
    \end{tikzcd}
\end{center}
we have again that the bottom right cospan is an equivalence, but now the top horizontal functor is only fully faithful, so that the left vertical functor is fully faithful.
\end{proof}

\begin{rem} Hopkins first claimed in \cite[Theorem 1']{MH} that the suspension functor $\Sigma: \cS_*\to coGrp_\infty(\cS_*^{\ge 2})$ is essentially surjective. However this result does not follow immediately from our work without further exploring the convergence of the Goodwillie tower of pointed spaces.  See also the discussion after Theorem 1.4 in \cite{KSV}.
\end{rem}

We end with a discussion on $n$-excisive approximations to functors. A reduced functor $F:\C\rightarrow \cD$ is 1-excisive if it takes a suspension diagram to a pullback. This tells us that the value of $F$ on $X$ can be recovered in $\cD$ from the suspension $F(X)\simeq \Omega F(\Sigma X)$.  In light of Theorem \ref{layer}, we see that a reduced functor is $n$-excisive when it can recover its values not from the suspension alone, but from the suspension equipped with its $A_n$-cogroup structure.  Examining the construction of the excisive approximation in Goodwillie's original work \cite{taylor} we see this is exactly the property being forced. A functor can be recovered from its tower when the value at any object can be recovered from the functors value on the entire  $A_\infty$-cogroup structure on the suspension.

\bibliographystyle{alpha}
\bibliography{ref}

\end{document}